\documentclass[12pt]{amsproc}
\usepackage[pdftex]{hyperref}
\newtheorem{theorem}{\sc Theorem}[section]
\newtheorem{lemma}[theorem]{\sc Lemma}
\newtheorem{proposition}[theorem]{\sc Proposition}

\begin{document}
\title[restricted centralizers of $\pi$-elements]{Profinite groups with restricted centralizers of $\pi$-elements}
\author{Cristina Acciarri}

\address{Cristina Acciarri: Dipartimento di Scienze Fisiche, Informatiche e Matematiche, Universit\`a degli Studi di Modena e Reggio Emilia, Via Campi 213/b, I-41125 Modena, Italy}
\email{cristina.acciarri@unimore.it}

\author{Pavel Shumyatsky }
\address{ Pavel Shumyatsky: Department of Mathematics, University of Brasilia,
Brasilia-DF, 70910-900 Brazil}
\email{pavel@unb.br}
\thanks{This research was supported by the Conselho Nacional de Desenvolvimento Cient\'{\i}fico e Tecnol\'ogico (CNPq),  and Funda\c c\~ao de Apoio \`a Pesquisa do Distrito Federal (FAPDF), Brazil.}
\begin{abstract} 
A group $G$ is said to have restricted centralizers if for each $g$ in $G$ the centralizer $C_G(g)$ either is finite or has finite index in $G$. Shalev showed that a profinite group with restricted centralizers is virtually abelian. Given a set of primes $\pi$, we take interest in profinite groups with restricted centralizers of $\pi$-elements. It is shown that such a profinite group has an open subgroup of the form $P\times Q$, where $P$ is an abelian pro-$\pi$ subgroup and $Q$ is a pro-$\pi'$ subgroup. This significantly strengthens a result from our earlier paper.

\end{abstract}

\maketitle

\section{Introduction}

A group $G$ is said to have restricted centralizers if for each $g$ in $G$ the centralizer $C_G(g)$ either is finite or has finite index in $G$. This notion was introduced by Shalev in \cite{shalev} where he showed that a profinite group with restricted centralizers is virtually abelian. We say that a profinite group has a property virtually if it has an open subgroup with that property. The article \cite{dms2} handles profinite groups with restricted centralizers of $w$-values for a multilinear commutator word $w$. The theorem proved in \cite{dms2} says that if $w$ is a multilinear commutator word and $G$ is a profinite group in which the centralizer of any $w$-value is either finite or open, then the verbal subgroup $w(G)$ is virtually abelian.  In \cite{AS}  we study profinite groups in which $p$-elements have restricted centralizers,  that is, groups in which $C_G(x)$ is either  finite or  open for any $p$-element $x$. The following theorem was proved.

\begin{theorem} \label{restricted}
Let $p$ be a prime and $G$ a profinite group in which the centralizer of each $p$-element is either finite or open. Then $G$ has a normal abelian pro-$p$ subgroup $N$ such that $G/N$ is virtually pro-$p'$.
\end{theorem}

The present paper grew out of our desire to determine whether this result can be extended to profinite groups in which the centralizer of each $\pi$-element, where $\pi$ is a fixed set of primes, is either finite or open. As usual, we say that an element $x$ of a profinite group $G$ is a $\pi$-element if the order of the image of $x$ in every finite continuous homomorphic image of $G$ is divisible only by primes in $\pi$ (see \cite[Section 2.3]{rz} for a formal definition of the order of a profinite group). 

It turned out that the techniques used in the proof of Theorem \ref{restricted} were not quite adequate for handling the case of $\pi$-elements. The basic difficulty stems from the fact that (pro)finite groups in general do not possess Hall $\pi$-subgroups.

In the present paper we develop some new techniques and establish the following theorem about finite groups.

If $\pi$ is a set of primes and $G$ a finite group, write $O^{\pi'}(G)$ for the unique smallest normal subgroup $M$ of $G$ such that $G/M$ is a $\pi'$-group. The conjugacy class containing an element $g\in G$ is denoted by $g^G$.

\begin{theorem}\label{main1} Let $n$ be a positive integer, $\pi$ be a set of primes, and $G$ a finite group such that $|g^G|\leq n$ for each $\pi$-element $g\in G$. Let $H=O^{\pi'}(G)$. Then $G$ has a normal subgroup $N$ such that
\begin{enumerate}
\item The index $[G:N]$ is $n$-bounded;
\item $[H,N]=[H,H]$;
\item The order of $[H,N]$ is $n$-bounded.
\end{enumerate}
\end{theorem}

Throughout the article we use the expression “$(a, b,\ldots)$-bounded” to mean that a quantity is finite and bounded by a certain number depending only on the parameters $a,b,\ldots$.

The proof of Theorem \ref{main1} uses some new results related to Neumann's BFC-theorem \cite{bhn}. In particular, an important role in the proof is played by a recent probabilistic result from \cite{DS}. Theorem \ref{main1} provides a highly effective tool for handling profinite groups with restricted centralizers of $\pi$-elements. Surprisingly, the obtained  result is much stronger than Theorem \ref{restricted} even in the case where $\pi$ consists of a single prime.

\begin{theorem}\label{main2} Let $\pi$ be a set of primes and $G$ a profinite group in which the centralizer of each $\pi$-element is either finite or open. Then $G$ has an open subgroup of the form $P\times Q$, where $P$ is an abelian pro-$\pi$ subgroup and $Q$ is a pro-$\pi'$ subgroup.
\end{theorem}
Thus, the improvement over Theorem \ref{restricted} is twofold -- the result now covers the case of $\pi$-elements and provides additional details clarifying the structure of groups in question. Furthermore, it is easy to see that Theorem \ref{main2} extends  Shalev's result \cite{shalev} which can be recovered by considering the case where $\pi=\pi(G)$ is the set of all prime divisors of the order of $G$.

We now have several results showing that if the elements of a certain subset $X$ of a profinite group $G$ have restricted centralizers, then the structure of $G$ is very special. This suggests the general line of research whose aim would be to determine which subsets of $G$ have the above property. At present we are not able to provide any insight on the problem. Perhaps one might start with the following question:
\medskip

\noindent {\it Let $n$ be a positive integer. What can be said about a profinite group $G$ such that if $x\in G$ then $C_G(x^n)$ is either finite or open?}
\medskip

Proofs of Theorems \ref{main1} and \ref{main2} will be given in Sections 2 and 3, respectively.

\section{Proof of Theorem \ref{main1}}

The following lemma is taken from \cite{AS}. If $X\subseteq G$ is a subset of a group $G$, we write $\langle X\rangle$ for the subgroup generated by $X$ and $\langle X^G\rangle$ for the minimal normal subgroup of $G$ containing $X$.

\begin{lemma}\label{dits}
Let $i,j$ be  positive integers and $G$ a group having a subgroup $K$ such that $|x^G|\leq i$ for each $x\in K$. Suppose that $|K|\leq j$. Then $\langle K^G\rangle$ has finite $(i,j)$-bounded order. 
\end{lemma}

If $K$ is a subgroup of a finite group $G$, we denote by 
$$Pr(K,G)=\frac{|\{(x,y) \in K\times G : [x,y]=1\}|}{|K||G|}$$
the relative commutativity degree of $K$ in $G$, that is, the probability that a random element of $G$ commutes with a random element of $K$. Note that 
$$Pr(K,G)=\frac{\sum_{x\in K}|C_G(x)|}{|K||G|}.$$
It follows that if $|x^G|\leq n$ for each $x\in K$, then $Pr(K,G)\geq\frac{1}{n}$.

The next result was obtained in \cite[Proposition 1.2]{DS}. In the case where $K=G$ this is a well known theorem, due to P. M. Neumann \cite{pneu}.

\begin{proposition}\label{rDS}
Let $\epsilon>0$, and let $G$ be a finite group having a subgroup $K$ such that $Pr(K,G)\geq \epsilon$. Then  there is a normal subgroup $T\leq G$ and a subgroup $B\leq K$ such that the indexes $[G:T]$ and $[K:B]$, and the order of the commutator subgroup $[T,B]$  are $\epsilon$-bounded.
\end{proposition}

We will now embark on the proof of Theorem \ref{main1}.

Assume the hypothesis of Theorem \ref{main1}. Let $X$ be the set of all $\pi$-elements of $G$. Clearly, $H=\langle X\rangle$. Given an element $g\in H$, we write $l(g)$ for the minimal number $l$ with the property that $g$ can be written as a product of $l$ elements of $X$. The  following result is straightforward from \cite[Lemma 2.1]{dieshu}.

\begin{lemma}\label{21} Let $K\leq H$ be a subgroup of index $m$ in $H$, and let $b\in H$. Then the coset $Kb$ contains an element $g$ such that $l(g)\leq m-1$.
\end{lemma}

Let $m$ be the maximum of indices of $C_H(x)$ in $H$ where $x\in X$. Obviously, we have $m\leq n$. 
\begin{lemma}\label{23} For any $x\in X$ the subgroup $[H,x]$ has $m$-bounded order.
\end{lemma}

\begin{proof} Take $x \in X$. Since the index of $C_H(x)$ in $H$ is at most $m$, by Lemma \ref{21}, we can choose elements $y_1,\ldots,y_m$ in $H$ such that $l(y_i)\leq m-1$ and the subgroup $[H,x]$ is generated by the commutators $[y_i,x]$, for $i=1,\ldots,m$. For any such $i$  write $y_i=y_{i1}\ldots y_{i(m-1)}$, with $y_{ij}\in X$. Using standard commutator identities we can rewrite $[y_i,x]$ as a product of conjugates in $H$ of the commutators $[y_{ij},x]$. Let $\{h_1,\ldots,h_s\}$ be the conjugates in $H$  of all elements from the set $\{x,y_{ij} \mid  1\leq i,j \leq m\}.$ Note that the number $s$ here is $m$-bounded. This follows form the fact that $C_H(x)$ has index at most $m$ in $H$ for each $x\in X$.  Put $T=\langle h_1,\ldots,h_s \rangle$. Since $[H,x]$ is contained in the commutator subgroup $T'$, it is sufficient to show that $T'$ has $m$-bounded order. Observe that the centre $Z(T)$ has index at most $m^s$ in $T$, since the index of $C_H(h_i)$ is at most $m$ in $H$ for any $i=1,\ldots,s$.  Thus, by Schur's theorem \cite[10.1.4]{Rob}, we conclude that the order of $T'$ is $m$-bounded, as desired. 
\end{proof}

Select $a\in X$ such that $|a^H|=m$. Choose $b_1,\ldots,b_m$ in $H$ such that $l(b_i)\leq m-1$ and $a^H=\{a^{b_i}; i=1,\ldots,m\}$. The existence of the elements $b_i$ is guaranteed by Lemma \ref{21}. Set $U=C_G(\langle b_1,\ldots,b_m\rangle)$. Note that the index of $U$ in $G$ is $n$-bounded. Indeed, since $l(b_i)\leq m-1$ we can write $b_i=b_{i1}\ldots b_{i(m-1)}$, where $b_{ij}\in X$ and $i=1,\ldots,m$. By the hypothesis the index of $C_G(b_{ij})$ in $G$ is at most $n$ for any such element $b_{ij}$. Thus, $[G:U]\leq n^{(m-1)m}$. 

The next result is somewhat analogous to \cite[Lemma 4.5]{wie}.
\begin{lemma}\label{24} If $u\in U$ and $ua\in X$, then $[H,u]\leq[H,a]$.
\end{lemma}
\begin{proof} Assume that $u\in U$ and $ua\in X$. For each $i=1,\ldots,m$ we have $(ua)^{b_i}=ua^{b_i}$, since $u$ belongs to $U$. We know that $ua\in X$ so taking into account the hypothesis on the cardinality of the conjugacy class of $ua$ in $H$, we deduce that $(ua)^H$ consists  exactly of the elements $ua^{b_i}$, for $i=1,\ldots,m$. Thus, given an arbitrary element $h\in H$, there exists $b\in \{b_1,\ldots,b_m\}$ such that $(ua)^h=ua^b$ and so $u^ha^h=ua^b$. It follows that $[u,h]=a^ba^{-h}\in[H,a]$, and the result holds.
\end{proof}

\begin{lemma}\label{derivedbounded}
The order of the commutator subgroup of $H$ is $n$-bounded.
\end{lemma}

\begin{proof} Let $U_0$ be the maximal normal subgroup of $G$ contained in $U$. Recall that, by the remark made before Lemma \ref{24}, $U$ has $n$-bounded index in $G$. It follows that the index $[G:U_0]$ is $n$-bounded as well. 

By the hypothesis $a$ has at most $n$ conjugates in $G$, say $\{a^{g_1},\ldots,a^{g_n}\}$. Let $T$ be the normal closure in $G$  of  the subgroup $[H,a]$. Note that  the subgroups $[H,a^{g_i}]$ are normal in $H$, therefore  $T=[H,a^{g_1}]\ldots [H,a^{g_n}]$. By Lemma \ref{23} each of the subgroups $[H,a^{g_i}]$ has $n$-bounded order.  We conclude that the order of $T$ is $n$-bounded. 

Let $Y=Xa^{-1}\cap U$. Note that for any $y\in Y$ the product $ya$ belongs to $X$.  Therefore, by Lemma \ref{24}, for any $y\in Y$, the subgroup $[H,y]$ is contained in $[H,a]$. Thus,
\begin{eqnarray}\label{111} [H,Y]\leq T.
\end{eqnarray} 

Observe that for any $u\in U_0$ the commutator $[u,a^{-1}]=a^ua^{-1}$ lies in $Y$ and so 
\begin{eqnarray}\label{112} 
[H,[U_0,a^{-1}]]\leq [H,Y]. 
\end{eqnarray} 

Since $[U_0,a^{-1}]=[U_0,a]$, we deduce from (\ref{111}) and (\ref{112}) that
\begin{eqnarray}\label{222} [H,[U_0,a]]\leq T.
\end{eqnarray}

Since $T$ has $n$-bounded order, it is sufficient to show that the derived group of the quotient $H/T$ has finite $n$-bounded order. We pass now to the quotient $G/T$ and for the  sake of simplicity the images of $G$, $H,U,U_0,X$ and  $Y$ will be denoted by the same symbols. Note that by (\ref{111}) the set $Y$ becomes central in $H$ modulo $T$. Containment (\ref{222}) shows that $[U_0,a]\leq Z(H)$.  This implies that if $b\in U_0$ is a $\pi$-element, then $[b,a] \in Z(H)$ and the subgroup $\langle a,b\rangle$ is nilpotent. Thus the product $ba$ is a $\pi$-element too and so $b\in Y$. Hence, all $\pi$-elements of $U_0$ are contained in $Y$ and, in view of (\ref{111}), we deduce that they are contained in $Z(H)$. 

Next we consider the quotient $G/Z(H)$. Since the image of $U_0$ in $G/Z(H)$ is a $\pi'$-group and has $n$-bounded index in $G$, we deduce that the order of any $\pi$-subgroup in $G/Z(H)$ is $n$-bounded. In particular, there is an $n$-bounded constant $C$ such that for every $p\in\pi$ the order of the Sylow $p$-subgroup of $G/Z(H)$ is at most $C$. Because of Lemma \ref{dits} for any $p\in \pi$  each Sylow $p$-subgroup of $G/Z(H)$ is contained in a normal subgroup of $n$-bounded order. We deduce that all such Sylow subgroups of $G/Z(H)$ are contained in a normal subgroup of $n$-bounded order. Since $H$ is generated by $\pi$-elements, it follows that the order of $H/Z(H)$ is $n$-bounded. Thus, in view of Schur's theorem \cite[10.1.4]{Rob}, we conclude that $|H'|$ is $n$-bounded, as desired.
\end{proof}

We will now complete the proof of Theorem \ref{main1}. 

\begin{proof}

Assume first that $H$ is abelian. In this case the set $X$ of $\pi$-elements is a subgroup, that is, $X=H$. By the hypothesis we have  $|x^G|\leq n$ for any element $x\in H$ and so the relative commutativity degree $Pr(H,G)$ of $H$ in $G$ is at least $\frac{1}{n}$. Thus, by virtue of Proposition \ref{rDS}, there is a normal subgroup $T\leq G$ and a subgroup $B\leq H$ such that the indexes $[G:T]$ and $[H:B]$, and the order of the commutator subgroup $[T,B]$ are $n$-bounded. 

Since $H$ is a normal $\pi$-subgroup and $[G:H]$ is a $\pi'$-number, by the Schur--Zassenhaus Theorem \cite[Theorem 6.2.1]{go} the subgroup $H$ admits a complement $L$ in $G$ such that $G=HL$ and $L$ is a $\pi'$-subgroup.  Set $T_0=T\cap L$. Observe that the index $[L:T_0]$ is $n$-bounded since it is at most the index of $T$ in $G$. Thus we deduce that the index of $HT_0$ is $n$-bounded in $G$, as well.  

We claim that the order of $[H,T_0]$ is $n$-bounded.  Indeed, the $\pi'$-subgroup $T_0$ acts coprimely on the the abelian  $\pi$-subgroup $B_1=B[B,T_0]$, and so we have $B_1=C_{B_1}(T_0)\times [B_1,T_0]$ (\cite[Corollary 1.6.5]{Kh}). Note that $[B_1,T_0]=[B,T_0]$. Since the oder of  $[B,T_0]$ is $n$-bounded (being at most the order of $[T,B]$),  we deduce that the index $[B_1:C_{B_1}(T_0)]$ is $n$-bounded. In combination with the fact that $[H:B]$ is $n$-bounded, we obtain that the index $[H:C_{B_1}(T_0)]$ is $n$-bounded and so in particular $[H:C_{H}(T_0)]$ is $n$-bounded. Since $T_0$ acts coprimely on  the abelian normal  $\pi$-subgroup $H$, we have $H=C_{H}(T_0)\times [H,T_0]$. Thus we obtain that the order of the commutator subgroup $[H,T_0]$ is $n$-bounded, as claimed. Let $T_1=C_{T_0}([H,T_0])$ and remark that the index $[T_0:T_1]$ of $T_1$ in $T_0$ is $n$-bounded too. Set $N=HT_1$. From the fact that the indexes $[T_0:T_1]$ and $[G:HT_0]$ are both $n$-bounded,  we deduce that  the index of $N$ in $G$ is $n$-bounded, as well.  

Note that $N$ is normal in $G$ since the image of $N$ in $G/H\cong L$ is isomorphic to $T_1$ which is normal in $L$. Furthermore,  we have $[H,T_1,T_1]=1$, since $T_1=C_{T_0}([H,T_0])$. Hence by the standard properties of coprime actions we have $[H,T_1]=1$ (\cite[Corollary 1.6.4]{Kh}). Therefore $[H,N]=1$. This proves the theorem in the particular case where $H$ is abelian.

In the general case, in view of Lemma \ref{derivedbounded}, the commutator subgroup $[H,H]$ is of $n$-bounded order. We pass to the quotient $\overline{G}=G/[H,H]$. The above argument shows that $\overline{G}$ has a normal subgroup $\overline{N}$ of $n$-bounded index such that $\overline{H}\leq Z(\overline{N})$. Here $Z(\overline{N})$ stands for the centre of $\overline{N}$. Let $N$ be the inverse image of $\overline{N}$. We have $[H,N]=[H,H]$ and so $N$ has the required properties. The proof is now complete.  
\end{proof}

\section{Proof of Theorem \ref{main2}}

We will require the following result taken from \cite[Lemma 4.1]{AS}.

\begin{lemma}\label{lemmaij}
Let $G$ be a locally nilpotent group containing an element with finite centralizer. Suppose that $G$ is residually finite. Then $G$ is finite.
\end{lemma}

Profinite groups have Sylow $p$-subgroups and satisfy analogues of the Sylow theorems.  Prosoluble groups satisfy analogues of the theorems on Hall $\pi$-subgroups. We refer the reader to the corresponding chapters in \cite[Ch.~2]{rz} and \cite[Ch.~2]{wil}.

Recall that an automorphism $\phi$ of a group $G$ is called fixed-point-free if $C_G(\phi)=1$, that is, the fixed-point subgroup is trivial. It is a well-known corollary of the classification of finite simple groups that if $G$ is a finite group admitting a fixed-point-free automorphism, then $G$ is soluble (see for example \cite{row} for a short proof). A continuous automorphism $\phi$ of a profinite group $G$ is coprime if for any open $\phi$-invariant normal subgroup $N$ of $G$ the order of the automorphism of $G/N$ induced by $\phi$ is coprime to the order of $G/N$. It follows that if a profinite group $G$ admits a coprime fixed-point-free automorphism, then $G$ is prosoluble. This will be used in the proof of Theorem \ref{main2}.

\begin{proof}[Proof of Theorem \ref{main2}] Recall that $\pi$ is a set of primes and $G$ is a profinite group in which the centralizer of every $\pi$-element is either finite or open. We wish to show that $G$ has an open subgroup of the form $P\times Q$, where $P$ is an abelian pro-$\pi$ subgroup and $Q$ is a pro-$\pi'$ subgroup. 

Let $X$ be the set of $\pi$-elements in $G$. Consider first the case where the conjugacy class $x^G$ is finite for any $x\in X$. For each integer $i\geq 1$ set $$S_i=\{x\in X;\ |x^G|\leq i\}.$$ The sets $S_i$ are closed.  Thus, we have countably many sets which cover the closed set $X$. By the Baire Category Theorem \cite[Theorem 34]{kel} at least one of these sets has non-empty interior. It follows that there is a positive integer $k$, an open normal subgroup $M$, and an element $a\in X$ such that all elements in $X\cap aM$ are contained in $S_k$.

Note that $\langle a^G\rangle$ has finite commutator subgroup, which we will denote by $T$. Indeed, the subgroup $\langle a^G\rangle $ is generated by finitely many elements whose centralizer is open. This implies that the centre of $\langle a^G\rangle $ has finite index in $\langle a^G\rangle $, and by Schur's theorem \cite[10.1.4]{Rob}, we conclude that $T$ is finite, as claimed. 

Let $x\in X\cap M$. Note that the product $ax$ is not necessarily in $X$. On the other hand, $ax$ is a $\pi$-element modulo $T$. This is because $\langle a^G\rangle $ becomes an abelian normal $\pi$-subgroup modulo $T$ and the image of $ax$ in the quotient $G/\langle a^G\rangle $ is a $\pi$-element. In other words, there are $y\in X\cap aM$ and $t\in T$ such that $ax=ty$. Observe that $t$ has an open centralizer in $G$ since $t\in T$. In fact $[G:C_G(t)]\leq|T|$. From the equality $ax=ty$ deduce that $|x^G|\leq k^2|T|$. This happens for any $x\in X\cap M$. Using a routine inverse limit argument in combination with Theorem \ref{main1} we obtain that $M$ has an open normal subgroup $N$ such that the index $[M:N]$ and the order of $[H,N]$ are finite. Here $H$ stands for the subgroup generated by all $\pi$-elements of $M$. Choose an open normal subgroup $U$ in $G$ such that $U\cap[H,N]=1$. Then $U\cap M$ is an open normal  subgroup of the form $P\times Q$, where $P$ is an abelian pro-$\pi$ subgroup and $Q$ is a pro-$\pi'$ subgroup. This proves the theorem in the case where all $\pi$-elements of $G$ have open centralizers.

Assume now that $G$ has a $\pi$-element, say $b$, of infinite order. Since the procyclic subgroup $\langle b\rangle$ is contained in the centralizer $C_G(b)$, it follows that $C_G(b)$ is open in $G$. This implies that all elements of $X\cap C_G(b)$ have open centralizers (because they centralize the procyclic subgroup $\langle b\rangle$). In view of the above $C_G(b)$ has  an open subgroup of the form $P\times Q$, where $P$ is an abelian pro-$\pi$ subgroup and $Q$ is a pro-$\pi'$ subgroup and we are done.
  
We will therefore assume that $G$ is infinite while all $\pi$-elements of $G$ have finite orders and there is at least one $\pi$-element, say $d$, such that $C_G(d)$ is finite. The element $d$ is a product of finitely many $\pi$-elements of prime power order. At least one of these elements must have finite centralizer. So without loss of generality we can assume that $d$ is a $p$-element for a prime $p\in\pi$. 

Let $P_0$ be a Sylow $p$-subgroup of $G$ containing $d$. Since $P_0$ is torsion, we deduce from Zelmanov's theorem \cite{zelmanov} that $P_0$ is locally nilpotent. The centralizer  $C_G(d)$ is finite and so in view of Lemma \ref{lemmaij} the subgroup $P_0$ is finite. Choose an open normal pro-$p'$ subgroup $L$ such that $L\cap C_G(d)=1$. Note that any finite homomorphic image of $L$ admits a coprime fixed-point-free automorphism (induced by the coprime action of $d$ on $L$). Hence $L$ is prosoluble. Let $K$ be a Hall $\pi$-subgroup of $L$. Since any element in $K$ has restricted centralizer, Shalev's result \cite{shalev} shows that $K$ is virtually abelian. We therefore can choose  an open normal subgroup $J$ in $L$ such that $J\cap K$ is abelian. If $J\cap K$ is finite then $G$ is virtually pro-$\pi'$ and we are done. If $J\cap K$ is infinite, then all $\pi$-elements of $J$ have infinite centralizers. This yields that  all $\pi$-elements of $J$ have open centralizers in $J$ and in view of the first part of the proof, $J$ has an open normal subgroup of the form $P\times Q$, where $P$ is an abelian pro-$\pi$ subgroup and $Q$ is a pro-$\pi'$ subgroup. This establishes the theorem.
\end{proof}

\end{document}